\documentclass{birkjour}

\usepackage{amssymb, latexsym, comment, amsmath, amsthm, upgreek, epsf, epsfig, color, verbatim, caption, tabularx}
\usepackage{amsfonts}
\usepackage{booktabs}
\usepackage{siunitx}

\usepackage{array}
\newcolumntype{L}[1]{>{\raggedright\let\newline\\\arraybackslash\hspace{0pt}}m{#1}}
\newcolumntype{C}[1]{>{\centering\let\newline\\\arraybackslash\hspace{0pt}}m{#1}}
\newcolumntype{R}[1]{>{\raggedleft\let\newline\\\arraybackslash\hspace{0pt}}m{#1}}
\makeatletter
\def\ps@pprintTitle{%
\let\@oddhead\@empty
\let\@evenhead\@empty
\def\@oddfoot{\centerline{\thepage}}%
\let\@evenfoot\@oddfoot}
\makeatother

\newtheorem{thm}{Theorem}[section]
\newtheorem{lemma}[thm]{Lemma}
\newtheorem{prop}[thm]{Proposition}

\newtheorem{defn}[thm]{Definition}
\newtheorem{rem}[thm]{Remark}
\newtheorem{example}[thm]{Example}

\begin{document}

\title[The G-invariant spectrum]{The G-invariant spectrum and non-orbifold singularities}

%\title[An Example for birkjour]
% {An Example for the Usage of the\\  birkjour Class File}

%----------Author 1
%\author[Birkh\"auser]{Birkh\"{a}user Publishing Ltd.}
\author{Ian M.~Adelstein and M.~R.~Sandoval} 
\address{Department of Mathematics, Trinity College\\ Hartford, CT 06106 United States}

%\email{info@birkhauser.ch}

\thanks{The authors would like to thank Carolyn Gordon and David Webb for many helpful conversations, as well as Emilio Lauret for providing valuable feedback. We would like to acknowledge the support of the NSF, Grant DMS-1632786. Finally, we would like to thank the referee for his or her thorough, prompt, and helpful review of the paper. }


\subjclass{58J50; 58J53; 22D99; 53C12}

\keywords{Spectral geometry; orbit spaces; group actions}

%\date{January 1, 2004}

\begin{abstract}
We consider the $G$-invariant spectrum of the Laplacian on an orbit space $M/G$ where $M$ is a compact Riemannian manifold and $G$ acts by isometries. We generalize the Sunada-Pesce-Sutton technique to the $G$-invariant setting to produce pairs of isospectral non-isometric orbit spaces. One of these spaces is isometric to an orbifold with constant sectional curvature whereas the other admits non-orbifold singularities and therefore has unbounded sectional curvature. We conclude that constant sectional curvature and the presence of non-orbifold singularities are inaudible to the G-invariant spectrum. 
\end{abstract}

%%% ----------------------------------------------------------------------
\maketitle
%%% ---------------------------------
%
%
%
%
%
%

\begin{comment}%%%%%%%
\begin{frontmatter}
\title{The G-invariant spectrum and non-orbifold singularities}
\author{Ian M.~Adelstein \\ M.~R.~Sandoval} 
\address{Department of Mathematics, Trinity College\\ Hartford, CT 06106 United States}
\begin{abstract} We consider the $G$-invariant spectrum of the Laplacian on an orbit space $M/G$ where $M$ is a compact Riemannian manifold and $G$ acts by isometries. We generalize the Sunada-Pesce-Sutton technique to the $G$-invariant setting to produce pairs of isospectral non-isometric orbit spaces. One of these spaces is isometric to an orbifold with constant sectional curvature whereas the other admits non-orbifold singularities and therefore has unbounded sectional curvature. We therefore show that constant sectional curvature and the presence of non-orbifold singularities are inaudible to the G-invariant spectrum. 
\end{abstract}
\begin{keyword} Spectral geometry, Laplace operator, orbit spaces, group actions
\MSC[2010] 58J50 \sep 58J53 \sep 22D99 \sep 53C12
\end{keyword}
\end{frontmatter}
\section*{Acknowledgements}
The authors would like to thank Carolyn Gordon and David Webb for many helpful conversations throughout the course of this project, as well as Emilio Lauret for providing valuable feedback. We would also like to acknowledge the support of the National Science Foundation, Grant DMS-1632786. Finally we would like to thank the referee for their thorough, prompt, and helpful review of the paper. 
\end{comment}%%%%%%%%

\section{Introduction} 

Given a compact Riemannian manifold $M$ and a compact subgroup of its isometry group $G \leq Isom(M)$ we consider the spectral geometry of the orbit space $M/G$. We are interested in the $G$-invariant spectrum:  the spectrum of the Laplacian restricted to the smooth functions on $M$ which are constant on the orbits. We let $C^{\infty}(M)^G$ denote the space of such functions. The orbit space $M/G$ has a natural metric structure that it inherits from $M$ and a natural smooth structure given by the algebra $C^{\infty}(M/G)$ consisting of functions $f \colon M/G \to \mathbb{R}$ whose pullback via $\pi \colon M \to M/G$ are smooth $G$-invariant functions on $M$, i.e.~$f \in C^{\infty}(M/G)$ if and only if $\pi^* f \in C^{\infty}(M)^G$. A map $\phi \colon M_1/G_1 \to M_2/G_2$ is said to be \emph{smooth} if the pullback of every smooth function on $M_2/G_2$ is a smooth function on $M_1/G_1$, i.e.~$\phi^*f  \in C^{\infty}(M_1/G_1)$ for every $f \in C^{\infty}(M_2/G_2)$. The notion of equivalence between orbit spaces that we use in this paper is the following: 

\begin{defn}\label{d:srfiso} A map $\phi \colon M_1/G_1 \to M_2/G_2$ is said to be a \emph{smooth SRF isometry} if it is an isometry of metric spaces that is smooth (in the above sense) with smooth inverse. 
\end{defn}

We note that the orbit space $M/G$ is an example of a singular Riemannian foliation and the term \emph{SRF isometry} comes from the literature on singular Riemannian foliations cf.~\cite{AS2017, AL2011}. If the orbit spaces $M_1/G_1$ and $M_2/G_2$ have the structure of Riemannian orbifolds then the above definition is equivalent to the standard notion of smooth isometry between Riemannian orbifolds. 

When $M$ is a compact Riemannian manifold one considers the eigenvalues of the Laplace-Beltrami operator $\Delta$, i.e.~those real numbers $\lambda$ for which there exists a non-trivial solution to the equation $\Delta(f) = \lambda f,~f \in C^{\infty}(M)$. These eigenvalues form a discrete sequence of non-negative real numbers. This sequence, counting multiplicities, is referred to as the spectrum of $M$. Given a compact subgroup of the isometry group $G \leq Isom(M)$ we consider the subsequence of eigenvalues that correspond to eigenfunctions which are constant on the $G$-orbits, again counting multiplicities. We will refer to this subsequence as the $G$-invariant spectrum of $M$. Given closed subgroups $H_i \leq G$ for $i \in \{1,2\}$ we say that the quotient spaces $M/H_1$ and $M/H_2$ are isospectral if the $H_i$-invariant spectra are equal. 

We are interested in the following inverse spectral questions: What information about the singular set of an orbit space $M/G$ is encoded in its $G$-invariant spectrum? In particular, can one hear the existence of non-orbifold singularities, i.e.~whether or not an orbit space is an orbifold? We note that the negative inverse spectral results from the manifold and orbifold settings hold in the more general setting of orbit spaces. It is therefore known that isotropy type \cite{SSW2006} and the order of the maximal isotropy groups \cite{RSW2008} are inaudible (i.e.~not determined by the $G$-invariant spectrum). Positive results exist for the $G$-equivariant spectrum where it has been shown that dimension and volume of the orbit space are spectrally determined \cite{BH1984}. In the Riemannian foliation setting it follows that the spectrum of the basic Laplacian determines dimension and volume of the space of leaf closures \cite{R1998}. 

In order to address the above spectral questions we first generalize the Sunada-Pesce-Sutton \cite{Sut2002} technique to the $G$-invariant setting (see Theorem~\ref{generalization}). We then use this generalization to produce pairs of isospectral non-isometric orbit spaces. As with all Sunada isospectral pairs our examples are quotients of a manifold $M$ by representation equivalent subgroups $H_1$ and $H_2$ of the isometry group. The quotients $M/H_i$ are isospectral in the sense that the $H_i$-invariant spectra of the Laplacian on $M$ are equal. 

\begin{thm}\label{main} Consider the subgroups $H_1 = U(3)$ and $H_2 = Sp(1) \times SO(4)$ of $U(6) \leq Isom(S^{11})$ under the embeddings where $A \in U(3)$ acts on $\mathbb{C}^6 = \mathbb{C}^3 \oplus \mathbb{C}^3$ as $(A, \bar{A})$ and $(B,C) \in Sp(1) \times SO(4)$ acts on $\mathbb{C}^6 = \mathbb{C}^2 \oplus \mathbb{C}^4$ as $(B,C)$, where $C$ acts on $\mathbb{C}^4 \cong \mathbb{R}^4 \oplus \mathbb{R}^4$ as $(C,C)$. Then the orbit spaces $S^{11}/H_1$ and $S^{11}/H_2$ are isospectral yet non-isometric.
\end{thm}

The representations of the groups $H_i$ used above are shown to be representation equivalent in \cite{AYY}. The authors then let these groups act on $SU(6)$ producing isospectral pairs of homogeneous Riemannian manifolds which are simply connected, yet non-homeomorphic \cite[Theorem 1.7]{AYY}. 

In Section~\ref{examples} we consider the geometry of the isospectral pair $O_1 = S^{11}/U(3)$ and $O_2 = S^{11}/(Sp(1) \times SO(4))$. We use principal isotropy reduction to study these spaces and show that they have different isotropy type, maximal isotropy dimension, and quotient codimension of the strata. The most striking differences between these isospectral orbit spaces are given in the following theorem. 

\begin{thm}\label{main2} The orbit space $O_1$ is smoothly SRF isometric to $S^3 / \mathbb{Z}_2$, a hemisphere of constant sectional curvature, whereas $O_2$ admits a non-orbifold point and therefore has unbounded sectional curvature. We conclude that constant sectional curvature and the presence of non-orbifold singularities are inaudible properties of the $G$-invariant spectrum.
\end{thm}

It is important to note that although the isometry between $O_1$ and $S^3 / \mathbb{Z}_2$ is a smooth SRF isometry, it does not preserve leaf codimension and we therefore can not conclude that these spaces are isospectral, cf.~\cite{AS2017}. Indeed, it can be shown by direct computation that the Neumann spectrum on $S^3 / \mathbb{Z}_2$ is distinct from the $U(3)$-invariant spectrum on $S^{11}$.

The question of whether orbifold singularities are audible, i.e.~whether there exists an isospectral pair of orbifolds one of which admits singular points whereas the other does not (and is therefore a manifold), is an important open question in the field of spectral geometry, cf.~\cite{Sut2010}. Our examples demonstrate that this question can be reformulated in the more general orbit space setting: Does there exist an isospectral pair of orbit spaces one of which admits singular points whereas the other does not (and is therefore a manifold)? 

The fact that constant sectional curvature is not determined by the $G$-invariant spectrum should be viewed in light of the positive inverse spectral results on constant sectional curvature in the manifold setting. By manipulating the terms in the asymptotic expansion of the heat trace, Berger et al.~\cite{Berger} showed that constant sectional curvature is an audible property of the Laplace spectrum for manifolds of dimension two or three. Using similar techniques Tanno \cite{Tanno1973} extended these results to manifolds of dimensions less than six. Although asymptotic expansions for the heat trace on orbifolds \cite{DGGW2008} and Riemannian foliations \cite{R2010} are known, the terms in these expansions are more complicated than in the manifold setting as they include information from the singular sets of these spaces. The methods of Berger et al.~and Tanno do not readily apply in these more exotic settings. Indeed, as the isospectral orbit spaces from Theorem~\ref{main} demonstrate, constant sectional curvature is an inaudible property of the G-invariant spectrum.

The paper proceeds as follows. In Section~\ref{sunada} we generalize the Sunada-Pesce-Sutton \cite{Sut2002} technique in order to prove Theorem~\ref{main}. In Section~\ref{smooth} we review a classification of orbifold points and orbifold singularities. In Section~\ref{examples} we study the inaudible properties of the isospectral non-isometric orbit spaces provided by Theorem~\ref{main} and prove Theorem~\ref{main2}.

\section{The $G$-invariant Sunada-Pesce-Sutton technique}\label{sunada}

Negative inverse spectral results are realized by studying pairs of isospectral non-isometric spaces. The celebrated Sunada technique \cite{Sun1985} provides a systematic method for producing such pairs. This technique has been generalized and applied in many different settings by modifying its various hypotheses. In this section we generalize the Sunada-Pesce-Sutton \cite{Sut2002} technique to the $G$-invariant setting. We start by reviewing the original Sunada technique and the various generalizations that lead to our $G$-invariant formulation. The original theorem concerns finite group actions and is stated in terms of \emph{almost conjugate} subgroups of the isometry group.

\begin{defn} Let $G$ be a finite group and let $H_1$ and $H_2$ be subgroups of $G$. Then $H_1$ is said to be \emph{almost conjugate} to $H_2$ in $G$ if $( \# [g]_G \cap H_1) =( \# [g]_G \cap H_2)  $ for each $G$-conjugacy class $[g]_G$. We call the triple $(G, H_1, H_2)$ a \emph{Sunada triple}. 
\end{defn}

\begin{thm}[Sunada's Theorem] Let $(G, H_1, H_2)$ be a Sunada triple and $M$ a compact Riemannian manifold on which $G$ acts by isometries. Assume that $H_1$ and $H_2$ act freely. Then $M / H_1$ and $M / H_2$ with the induced Riemannian metrics are strongly isospectral as Riemannian manifolds. 
\end{thm} 

A simple generalization of this theorem allows for the action to have fixed points (i.e.~not necessarily free). Using this generalization one can produce isospectral pairs of orbifolds, cf.~\cite[Theorem 2.5]{SSW2006}. It is precisely the hypothesis that the action be free that we will modify in the Sunada-Pesce-Sutton generalization to produce isospectral pairs of orbit spaces. Such pairs arise as quotients of a compact Riemannian manifold by nonfinite subgroups of its isometry group. The \emph{almost conjugate} condition no longer applies and we instead appeal to the notion of \emph{representation equivalence} of subgroups. 

\begin{defn} Two representations $\rho_1 \colon G \to GL(V_1)$ and $\rho_2 \colon G \to GL(V_2)$ of a Lie group $G$ are said to be \emph{equivalent} if there exists a vector space isomorphism $T \colon V_1 \to V_2$ such that $\rho_2(g) \circ T = T \circ \rho_1(g)$ for every $g \in G$.
\end{defn}

\begin{defn} Two closed subgroups $H_1, H_2 \leq G$ of a compact Lie group $G$ are said to be \emph{representation equivalent} if the quasi-regular representations $\emph{Ind}_{H_1}^G(1_{H_1})$ and $\emph{Ind}_{H_2}^G(1_{H_2})$ are equivalent.
\end{defn}

Note when $G$ is a finite group that subgroups $H_1$ and $H_2$ are almost conjugate if and only if they are representation equivalent. 

One of the earliest and most fruitful generalizations of the original Sunada theorem was provided by DeTurck and Gordon \cite{DeTGL}. They allow $G$ to be a Lie group and $H_1$ and $H_2$ to be representation equivalent cocompact discrete subgroups. When $G$ acts by isometries on a Riemannian manifold $M$ in such a way that $M / H_1$ and $M / H_2$ are compact manifolds then these manifolds are strongly isospectral. 

Pesce \cite{P} weakened the hypothesis that the discrete subgroups $H_1$ and $H_2$ be representation equivalent, applying the idea of Frobenius reciprocity in his proof. Sutton \cite[Theorem 2.3]{Sut2002} generalized the Sunada-Pesce technique from the setting of discrete subgroups to connected subgroups, allowing for the construction of the first locally non-isometric, simply-connected isospectral homogeneous manifolds.

We now state the $G$-invariant version of the Sunada theorem that we will apply in this paper, noting that the proof follows directly the proofs of the Pesce-Sutton generalizations. We let $[ \rho_1 : \rho_2]$ denote the multiplicity of the representation $\rho_2 \colon G \to GL(V_2)$ in $\rho_1 \colon G \to GL(V_1)$.

\begin{thm}[$G$-invariant Sunada-Pesce-Sutton technique]\label{generalization} Let $M$ be a compact Riemannian manifold and $G \leq Isom(M)$ a compact Lie group. Suppose $H_1, H_2 \leq G$ are closed, representation equivalent subgroups. Then the orbit spaces $M/H_1$ and $M/H_2$ are isospectral in the sense that the $H_i$-invariant spectra of the Laplacian on $M$ are equal. 
\end{thm}

\begin{proof} For each eigenspace $E_{\lambda}$ of the Laplacian on $C^{\infty}(M)$ we must show that the dimension of the space of $H_1$ fixed vectors equals that of the $H_2$ fixed vectors. The group $G$ acts on $E_{\lambda}$; denote this action $\rho_{\lambda} \colon G \to Hom(E_{\lambda})$. We have that $\text{dim}(E_{\lambda}^{H_i}) = [\text{Res}_{H_i}^G ( \rho_{\lambda}) : 1_{H_i} ] $ and by Frobenius reciprocity that $[\text{Res}_{H_i}^G ( \rho_{\lambda}) : 1_{H_i} ] = [\text{Ind}_{H_i}^G ( H_i ) : \rho_{\lambda} ] $. By the representation equivalence of the $H_i \leq G$ we have $[\text{Ind}_{H_1}^G ( H_1 ) : \rho_{\lambda} ] = [\text{Ind}_{H_2}^G ( H_2) : \rho_{\lambda} ]$. We conclude that the $H_1$-invariant spectrum is equal to the $H_2$-invariant spectrum. 
\end{proof}

Sutton generalizes his own theorem in \cite[Theorem 2.7]{Sut2010} and the version stated above is a special case of this result.

\begin{proof}[Proof of Theorem~\ref{main}] Under the given embeddings \cite[Theorem 1.5]{AYY} demonstrates that $H_1$ and $H_2$ are representation equivalent as subgroups of $SU(6) \leq Isom(S^{11})$. The isospectrality of the spaces $O_1 = S^{11}/H_1$ and $O_2 = S^{11}/H_2$ then follows immediately from Theorem~\ref{generalization}. Finally, we note that there are many geometric properties that distinguish these spaces as non-isometric. We catalogue these properties in Section~\ref{examples}. 
\end{proof}

%
%
%
%%%%%%%%

\begin{section}{Orbifold Points and Non-orbifold Points in an Orbit Space}\label{smooth}
Here we make note of an important characterization of orbifolds in terms of slice representations from \cite{LT2010} that we will use to prove Theorem \ref{main2}. We begin by recalling the notion of polar actions and polar representations, comprehensively studied in \cite{D1985}.

\begin{defn}
An action by a Lie group $G$ of isometries on a complete Riemannian manifold $M$ is said to be {\it polar} if there exists a complete submanifold $\Sigma$ that meets every orbit orthogonally. Such a submanifold is known as a {\it section}. Note that such submanifolds are necessarily totally geodesic.
\end{defn}

\begin{example}\label{s2}\normalfont
By way of illustration of this idea, we note that the action of $SO(2)$ on $S^2$ by rotations about an axis is polar. One may choose $\Sigma$ to be any meridian great circle. The quotient space is an interval of the form $[0,\pi].$ This has an orbifold structure with $\mathbb{Z}_2$ action at the endpoints. In fact, it is a good orbifold, $S^1/\mathbb{Z}_2.$ This is more generally true for all polar actions, as we see from the proposition below, which follows easily from Chapter 5 of \cite{PT1988}.
\end{example}

\begin{prop}
Let $M$ be a compact Riemannian manifold and $G$ a closed subgroup of the isometry group of $M$. If the action of $G$ on $M$ is polar, then the quotient $M/G$ has the smooth structure of a good orbifold.
\end{prop}

\begin{proof} Let $\Sigma\subset M$ be a section. Now recall the generalized Weyl group of $\Sigma$ from \cite{PT1988}: Let $N(\Sigma)=\{g\in G\,|\,g\cdot\Sigma=\Sigma\}$. Note that this is the largest subgroup of $G$ which induces an action on $\Sigma$. Let $Z(\Sigma)=\{g\in G\,|\,g\cdot s=s,\, \forall s\in \Sigma\} ,$ which is the kernel of the induced action. Then the {\it generalized Weyl group of }$\Sigma$ is $W(\Sigma)=N(\Sigma)/Z(\Sigma)$. Let $W:=W(\Sigma).$ From Proposition 5.6.15 of \cite{PT1988}, $W$ is discrete in general, and thus, for $G$ compact, $W$ is finite. Thus $\Sigma/W$ has the smooth structure of a good orbifold. By Theorem 5.6.25 of \cite{PT1988}, the restriction map $r: C^\infty(M)^G\rightarrow C^\infty(\Sigma)^W$ is an isomorphism, and hence $M/G$ also has the smooth structure of a good orbifold.
\end{proof}

\begin{rem}\normalfont
Notice that in Example \ref{s2}, the action of $S^1$ is effective and thus $W$ is the group generated by 180-degree rotations, and hence is isomorphic to $\mathbb{Z}_2$ with the previously claimed action on the endpoints.
\end{rem}

The arguments in the previous proof are global, rather than local, in nature and produce a stronger result than we seek--that of a good orbifold. One can ask if weaker, more local conditions on the action would guarantee a more generic orbifold structure in the quotient. Indeed, that is known to be the case in the literature of orbit spaces, as we see in what follows.

%%%%%%%%%%%

\begin{defn}
Let $x\in M$ and let $H_x$ be the normal space to the orbit through $x$. The isotropy group $G_x$  acts on $H_x.$ This is called the {\it slice representation} at $x$. If the slice representations at every point $x\in M$ are polar, then the representation is said to be {\it infinitesimally polar}, as in \cite{GL2015}. We note that every polar representation is infinitesimally polar.
\end{defn}

The characterization of orbifold points that we will use comes from Theorem 1.1 of \cite{LT2010}, which we state below:

\begin{thm}[Theorem 1.1 of \cite{LT2010}]\label{22}
Let $M$ be a Riemannian manifold and let $G$ be a closed group of isometries of $M$. Let $B=M/G$ be the quotient, which is stratified by orbit type. Let $B_0$ be the maximal stratum of $B$. Let $x\in M$ be a point with isotropy group $G_x$ acting on the normal space $H_x$ of the orbit $G\cdot x\subset M$. Set $y=G\cdot x\in B.$ Let $z\in B_0$, and let $\overline{\kappa}(z)$ be the maximum of the sectional curvatures of the tangent planes at $z$. Then the following are equivalent:
\begin{enumerate}
\item $\displaystyle{\limsup_{z\in B_0, z\rightarrow y}\overline{\kappa}(z)<\infty.}$
\item $\displaystyle{\limsup_{z\in B_0, z\rightarrow y}\overline{\kappa}(z)\cdot (d(z,y))^2=0.}$
\item The action of $G_x$ on $H_x$ is polar. 
\item A neighborhood of $y$ in $B$ is a smooth Riemannian orbifold. 
\end{enumerate}
\end{thm}

It is these last two equivalent conditions that are of particular interest here. As a consequence, an orbit space has an orbifold structure if and only if the action of $G$ on $M$ is infinitesimally polar. We note that the proof of the last two equivalences above relies on fundamental properties and results relating to singular Riemannian foliations and polar actions, and thus illustrates how one may augment the local understanding of orbifolds via the connection with quotients of manifolds by isometric group actions (or, equivalently, with the leaf spaces of singular Riemannian foliations with closed leaves).

\begin{defn}
A point $y$ in the quotient satisfying the last condition of the above is said to be an {\it orbifold point}. Otherwise, $y$ is a non-orbifold point.
\end{defn}
As noted in Corollary 1.2 of \cite{LT2010}, the set of non-orbifold points in the quotient has relatively high quotient codimension--at least 3.

\end{section}

\section{The Examples}\label{examples}

In this section we examine the pair of isospectral spaces $O_1 = S^{11}/U(3)$ and $O_2 = S^{11}/(Sp(1) \times SO(4))$ provided by Theorem~\ref{main} to determine inaudible properties of the $G$-invariant spectrum. In particular, we prove Theorem~\ref{main2} demonstrating that the existence of non-orbifold singularities and constant sectional curvature are not determined by the $G$-invariant spectrum. Additionally we have the following proposition which follows immediately from the tables below.

\begin{prop}\label{props} As demonstrated by the isospectral pair $O_1 = S^{11}/U(3)$ and $O_2 = S^{11}/(Sp(1) \times SO(4))$ the following properties of orbit spaces are inaudible, i.e.~not determined by the $G$-invariant spectrum:~isotropy type, maximal isotropy dimension, and the set of quotient codimensions of the strata. 
\end{prop}

We note that some of these results are known in the orbifold setting where the inaudibility of isotropy type \cite{SSW2006} and the order of the maximal isotropy groups \cite{RSW2008} have been demonstrated. Quotient codimension of the strata is a topological property of the quotient and is known to be preserved under smooth SRF isometry \cite{AS2017}. It is nevertheless not determined by the $G$-invariant spectrum. 

The first step in our study of the pairs of isospectral non-isometric orbit spaces provided by Theorem~\ref{main} is an application of principal isotropy reduction, cf.~\cite{GL2015} or \cite{GL2014}.

\begin{lemma}\label{pid} Principal isotropy reduction yields the following smooth SRF isometries (note that these isometries do not preserve the spectra): $$O_1= S^{11}/U(3)=S^7/U(2)$$$$O_2= S^{11}/(Sp(1) \times SO(4))= S^7/(Sp(1)\times O(2)).$$
\end{lemma}

\begin{proof}

Here we view $S^{11} \subset \mathbb{C}^6$ so that points $(z_1, \ldots, z_6) \in S^{11}$ are such that $z_i \in \mathbb{C}$ and $|z_1|^2 + \dots + |z_6|^2 = 1$. We consider each of the reductions separately. First we apply principal isotropy reduction to $S^{11}/U(3)$. The principal isotropy for this action is $K=U(1)\times Id_2$. We choose a specific copy of $K \leq U(3)$, say

\[
K= \bigg{ \{ }
\left[
\begin{array}{c | c c}
z & &  \\ 
\hline 
 & 1 & 0 \\
 & 0 & 1 \\
 \end{array}
 \right] : z \in U(1)   \bigg{  \} }
\]

\noindent and then use the embedding of $U(3)$ into $Isom(S^{11})$ given in Theorem~\ref{main} to realize the principal isotropy as 

\[
K= \bigg{ \{ }
\left[
\begin{array}{ c | c}
\begin{array}{c | c c}
z & &  \\ 
\hline 
 & 1 & 0 \\
 & 0 & 1 \\
 \end{array} & \\
\hline 
& \begin{array}{c | cc }
 \bar{z} & &  \\ 
\hline 
  & 1 & 0 \\
  & 0 & 1 \\ 
\end{array}
\end{array}
\right]   \colon z \in U(1)     \bigg{ \} }       \leq SU(6) \leq Isom(S^{11}).
\]

\noindent We see that the vectors $v=(0,z_2,z_3,0,z_5,z_6) \in S^{11}$ are fixed by $K$ and that the fixed point set is therefore a copy of $S^7 \subset S^{11}$. Let $N$ denote the induced action on this fixed point set of the normalizer of $K$ in $U(3)$. Because $N$ acts on the vectors $v=(0,z_2,z_3,0,z_5,z_6) \in S^{11}$ the largest it can be is

\[
N= \bigg{ \{ }
\left[
\begin{array}{ c | c}
\begin{array}{c | c  }
z &   \begin{array}{c c}  &  \end{array}  \\ 
\hline 
  \begin{array}{c }  \\  \end{array} & \\ & A
\end{array} &  \\
\hline 
& \begin{array}{c | c }
\bar{z} &   \begin{array}{c c} &  \end{array}  \\ 
\hline 
  \begin{array}{c }  \\  \end{array}  & \\ &  \bar{A}
\end{array}
\end{array}
\right]  : z \in U(1) , ~ A \in U(2)   \bigg{  \} }.
\]

\noindent It is straightforward to check that $U(1) \times U(2)$ normalizes $K=U(1) \times Id_2$ so that indeed $N=U(1) \times U(2)$. Principal isotropy reduction therefore yields the metric space isometry $S^{11}/U(3)= S^7/ (N/K) = S^7/U(2)$. Finally, as these orbit spaces are all three dimensional, \cite[Theorem 1.5]{AL2011} implies that these metric space isometries are smooth SRF isometries.

%NEXT EXAMPLE

Next we apply principal isotropy reduction to $S^{11}/(Sp(1) \times SO(4))$. The principal isotropy for this action is $K=Id_4 \times (SO(2) \times Id_2)$. We choose a specific copy of $K \leq Sp(1)\times SO(4)$, say

\[
K= \bigg{ \{ }
\left[
\begin{array}{c | c  c }
Id_4 & &  \\ 
\hline 
 & A & 0   \\
 & 0 & Id_2 \\
  \end{array}
\right]  : A \in SO(2) \bigg{ \} }
\]

\noindent and then use the embedding of $Sp(1) \times SO(4)$ into $Isom(S^{11})$ given in Theorem~\ref{main} to realize the principal isotropy as

\[
K=  \bigg{ \{ }
\left[
\begin{array}{c | c  c | c c}
Id_4 & & & &  \\ 
\hline 
 & A & 0 & &  \\
 & 0 & Id_2 & & \\
 \hline
 & & & A & 0 \\
 & & & 0 & Id_2 \\
 \end{array}
\right]   : A \in SO(2) \bigg{ \} }  \leq SU(6) \leq Isom(S^{11}).
\]

\noindent We see that the vectors $v=(z_1, z_2, 0, z_4, 0, z_6) \in S^{11}$ are fixed by $K$ and that the fixed point set is therefore a copy of $S^7 \subset S^{11}$. Let $N$ denote the induced action on this fixed point set of the normalizer of $K$ in $Sp(1) \times SO(4)$. Because $N$ acts on the vectors $v=(z_1, z_2, 0, z_4, 0, z_6) \in S^{11}$ the largest it can be is

\[
N=  \bigg{ \{ }
\left[
\begin{array}{c | c  c | c c}
B & & & &  \\ 
\hline 
 & C & 0 & &  \\
 & 0 & D & & \\
 \hline
 & & & C & 0 \\
 & & & 0 & D \\
 \end{array}
\right]   : B \in Sp(1),  \left[  \begin{array}{ c c} C & 0 \\ 0 & D \end{array} \right]  \in SO(4)       \bigg{ \} }
\]

\noindent It is straightforward to check that the above matrices normalize $K=Id_4 \times (SO(2) \times Id_2)$ so that indeed $N= Sp(1) \times S(O(2) \times O(2))$. We have that $N/K = Sp(1) \times S(O(1) \times O(2)) = Sp(1) \times O(2)$ so that principal isotropy reduction yields the metric space isometry $$S^{11}/(Sp(1) \times SO(4))= S^7/(N/K)=  S^7/(Sp(1) \times O(2))  .$$ Again, as these orbit spaces are all three dimensional, \cite[Theorem 1.5]{AL2011} implies that these metric space isometries are smooth SRF isometries. 
\end{proof}

\begin{rem}\normalfont Arguments similar to those of Theorem~\ref{main} show that the orbit spaces $S^{4n-1}/U(n)$ and $S^{4n-1}/(Sp(m) \times SO(2n-2m))$ are isospectral for odd integers $n \geq 3$ with $m=(n-1)/2$. We note, however, that one can not identify additional inaudible properties from these pairs as principal isotropy reduction yields the following smooth SRF isometries (again note that these isometries do not preserve the spectra): $$S^{4n-1}/U(n)=S^7/U(2)$$$$S^{4n-1}/(Sp(m) \times SO(2n-2m))= S^7/(Sp(1)\times O(2)).$$
\end{rem}

We now study the isotropy type and quotient codimension of these reduced spaces. We note that the embeddings of $U(2)$ and $Sp(1)\times O(2)$ into $U(4) \leq Isom(S^7)$ are analogous to the embeddings given in Theorem~\ref{main}. We therefore consider the vector decomposition $v=(v_1,v_2) \in S^7 \subset \mathbb{C}^2 \oplus  \mathbb{C}^2$ when discussing the first quotient and $v=(v_1,v_2,v_3) \in S^7 \subset  \mathbb{C}^2 \oplus \mathbb{C} \oplus \mathbb{C}$ when discussing the second quotient. The following two tables catalogue the isotropy and quotient codimension of these reduced spaces using this vector decomposition.

\begin{center}
\captionof{table}{$O_1 =S^7/U(2)$}
\begin{tabular}{ l L{2.8cm} c L{3.5cm}} \toprule
    {Row} &  {Isotropy} & {qcodim} & {Points}  \\ \midrule
    $A$ & $Id$   & 0 & $v_1 \ne z \cdot \bar{v}_2$    \\
    $B$ & $U(1)$ & 1 &  $v_1= z \cdot \bar{v}_2$  \\  \bottomrule
    %$B_2$ & $U(1)$ & $3$ & $v_1 \ne 0,~ v_2=0$ \\ 
    %$B_3$ & $U(1)$ & $3$ & $v_1=0, ~v_2 \ne 0$ \\   \bottomrule
\end{tabular}
$\text{Note: } v=(v_1,v_2) \in S^7 \subset \mathbb{C}^2 \oplus  \mathbb{C}^2$ and $z \in \mathbb{C}$
\end{center}

%\pagebreak

\vspace{.1in}

\begin{center}
\captionof{table}{$O_2 =S^7/(Sp(1)\times O(2))$}
\begin{tabular}{l L{3.3cm} c L{3.5cm}} \toprule
     {Row} & {Isotropy} & {qcodim} & {Points}  \\ \midrule
     $A$ & $Id \times Id$ & 0 &  $v_1 \ne 0, ~ v_2 \ne \lambda \cdot v_3$  \\
     $B$ & $Id \times O(1)$  & 1 &  $v_1 \ne 0, ~ v_2 = \lambda \cdot v_3$   \\
     %$B_2$ & $Id \times O(1)$  & 2 &  $v_1 \ne 0, ~ v_2 \ne 0,~ v_3= 0$   \\
     %$B_3$ & $Id \times O(1)$  & 2 &  $v_1 \ne 0, ~ v_2 =0,~ v_3 \ne 0$   \\
     $C$ & $Sp(1) \times Id$ & 1 & $v_1 = 0, ~ v_2 \ne \lambda \cdot v_3$   \\
     $D$ & $Id \times O(2) $ & 3 & $v_1 \ne 0, ~ v_2 =v_3=0$    \\
     $E$ & $Sp(1) \times O(1)$ & 2 & $v_1 = 0, ~ v_2 = \lambda \cdot v_3$  \\ \bottomrule
     %$E_2$ & $Sp(1) \times O(1)$ & 3 & $v_1 = 0, ~ v_2 \ne 0,~ v_3= 0$ \\ 
     %$E_3$ & $Sp(1) \times O(1)$ & 3 & $v_1 = 0, ~ v_2 =0,~ v_3 \ne 0$ \\ \bottomrule
\end{tabular}
$\text{Note: } v=(v_1,v_2,v_3) \in S^7 \subset  \mathbb{C}^2 \oplus \mathbb{C} \oplus \mathbb{C}$ and $\lambda \in \mathbb{R}$
\end{center}

\begin{lemma}\label{two} The space $O_2=S^7/(Sp(1)\times O(2))$ admits a non-orbifold point.
\end{lemma}

\begin{proof} We first show that the slice representation of the action is non-polar at points $v =(v_1,0,0) \in S^7$ from row D of Table 2. By Theorem~\ref{22} we can then conclude that the image of this stratum is a non-orbifold point. 

The slice representation of the action at $v$ is polar if and only if the slice representation of the restriction of the action to the connected component of the identity is polar, cf.~\cite[Section 2.4]{GL2015}. We therefore consider the action of $Sp(1)\times SO(2)$ on $S^7$ which acts with isotropy $Id \times SO(2)$ at $v =(v_1,0,0) \in S^7$. We recall from \cite[Lemma 2]{GL2015} that polar actions of connected groups with trivial principal isotropy must have a codimension 1 stratum with non-trivial isotropy. The slice representation of $Sp(1)\times SO(2)$ at $v =(v_1,0,0) \in S^7$ has non-trivial isotropy only at the zero vector, hence is non-polar.
\end{proof}

\begin{lemma}[\cite{GL2015}, Theorem 1]\label{one} The space $O_1= S^{11}/U(3)$ is isometric to $S^3 / \mathbb{Z}_2$, the 3-hemisphere of constant sectional curvature 4. 
\end{lemma} 

The proof of Lemma~\ref{one} is provided in \cite[Section 3.4]{GL2015} where $S^{11}/U(3)$ is first reduced to $S^7/U(2)$ via principal isotropy reduction as above and then shown to be isometric to $S^3 / \mathbb{Z}_2$, the 3-hemisphere of constant curvature 4. We note that O'Neill's formula can also be used to show that the space $S^7/U(2)$ has constant sectional curvature 4.

\begin{proof}[Proof of Theorem~\ref{main2}] Lemmas~\ref{two} and \ref{one} show that the orbit space $O_1$ is isospectral to an orbifold whereas $O_2$ admits a non-orbifold point, demonstrating the inaudibility of non-orbifold singularities. We can then apply Theorem~\ref{22} to the space $O_2$ to conclude that it has unbounded sectional curvature. We conclude that constant sectional curvature is not determined by the $G$-invariant spectrum. 
\end{proof}

%%%%%%%%%%%%%%
%END SECTION 4
%%%%%%%%%%%%%%

%%%%%%%%%%%%%%
%END SECTION 4
%%%%%%%%%%%%%%
%
%
%
%
%

%\section*{References}

%\bibliographystyle{abbrv}
%\bibliography{sandoval}
\bibliographystyle{alpha} 
%\bibliography{biglist}

\end{document}